\documentclass{amsart}
\usepackage{delimset, amssymb, enumitem, caption, xcolor}
\setlist{itemsep=4pt, topsep=0pt, leftmargin=17pt, listparindent=11pt}

\usepackage[colorlinks, citecolor=cyan, linkcolor=magenta, pagebackref, urlcolor=brown, ocgcolorlinks]{hyperref}
\usepackage{tikz}
\tikzset{edge/.style={very thick, color=black}}
\tikzset{ball/.style={shape=circle, inner sep=9pt, ball color=black}}
\tikzset{apball/.style={ball color=black}}

\newcommand\ind{\mathord{\uparrow}}

\usepackage[capitalize]{cleveref}
\crefname{conjecture}{Conjecture}{Conjectures}
\newtheorem{theorem}{Theorem}[section]
\newtheorem{corollary}[theorem]{Corollary}
\newtheorem{lemma}[theorem]{Lemma}
\newtheorem{conjecture}[theorem]{Conjecture}
\newtheorem{proposition}[theorem]{Proposition}

\numberwithin{equation}{section}
\numberwithin{figure}{section}

\allowbreak
\allowdisplaybreaks

\usepackage[numbers, sort&compress, nonamebreak, merge, elide, longnamesfirst]{natbib}

\author[D.G.L.~Wang]{David G.L. Wang}
\address[David G.L. Wang]{School of Mathematics and Statistics \& Beijing Key Laboratory on MCAACI, Beijing Institute of Technology; MIIT Key Laboratory of Mathematical Theory and Computation in Information Security, Beijing 102400, P. R. China}
\email{glw@bit.edu.cn}

\author[M.M.Y. Wang]{Monica M.Y. Wang$^*$}
\address[Monica M.Y. Wang]{School of Mathematics and Statistics, Beijing Institute of 
Technology, Beijing 102400, P. R. China}
\email{mengyu919@bit.edu.cn}



\thanks{Corresponding author: Monica M.Y. Wang. 
The first author is supported by the General Program of National Natural 
Science Foundation of China (Grant No.~12171034) and 
the Fundamental Research Funds for the Central Universities in China (Grant No.~2021CX11012).
The second author is supported by the Training Funds for Excellent Doctoral Thesis of Beijing Institute of Technology in China.}

\title[]{Two cycle-chord graphs are $e$-positive}

\begin{document}
\bibliographystyle{abbrvnat}
\maketitle
\begin{abstract}
We prove Gebhard and Sagan's $(e)$-positivity of the line graphs of tadpoles in noncommuting variables. This implies the $e$-positivity of these line graphs. We then 
extend this $(e)$-positivity result to that of certain cycle-chord graphs, and derive the bivariate generating function of all cycle-chord graphs.
\end{abstract}

\section{Introduction}

\citet{Sta95} introduced 
the {\it chromatic symmetric function}
for a simple graph $G$ as
\[
X_G
=\sum_{\kappa}\prod_{v\in V(G)}\mathbf{x}_{\kappa(v)}
\]
where $\mathbf{x}=(x_1, x_2, \ldots)$ is a countable set
of indeterminates, 
and the sum is over all proper colorings~$\kappa$ of $G$.
Chromatic symmetric functions is a generalization of \citeauthor{Bir12}'s chromatic polynomials $\chi_G(k)$ since
\[
X_G(1^k,0,0,\dots)=\chi_G(k),
\]
see \citet{Bir12}.

Chromatic symmetric functions are particular symmetric functions.
It is the fundamental theorem of symmetric functions that $\{e_\lambda\}$ is a basis of the algebra $\Lambda(x_1,x_2,\ldots)$ of symmetric functions, where
\[
e_\lambda
=e_{\lambda_1}e_{\lambda_2}\dotsm \quad\text{and}\quad
e_n
=\sum_{1\le i_1<\dots<i_n}x_{i_1}\dotsm x_{i_n}.
\]
The algebra $\Lambda(x_1,x_2,\ldots)$ has also some other bases such like the Schur basis $\{s_\lambda\}$.
For any bases $\{b_\lambda\}$,
a symmetric function is \emph{$b$-positive}
if its expansion under the basis~$b_{\lambda}$ has only nonnegative coefficients, see \citet{Mac95B} and \citet[Chapter 7]{Sta99B}.
A graph is said to be $b$-positive if its chromatic symmetric function is $b$-positive.
A class of graphs is said to be $b$-positive if every graph in the class is $b$-positive.

This work is originally motivated by Stanley and Stembridge's 3+1 conjecture, see~\citet{SS93}.

\begin{conjecture}[\citeauthor{SS93}]\label{conj:ep:cf-inc}
Any claw-free incomparability graph is $e$-positive.
\end{conjecture}

Only a few methods are known to show the $e$-positivity of a graph, while there are many results on the $e$-positivity of  particular graph classes.
\citet{Wol97D} provided a powerful criterion that
every connected $e$-positive graph has a connected partition of every type.
Graph classes that are shown to be $e$-positive include
complete graphs, paths, cycles, 
generalized bull graphs, 
$K$-chains, lollipop graphs, triangular ladders,
Ferrers graphs;
see \cite{Cv16, Sta95, CH18, CH19, FHM19, Dah19, Dv18, GS01, Ev04,  LY21}.
Graphs that are proved not to be $e$-positive include
generalized nets,
saltire graphs $\mathrm{SA}_{n,n}$,
augmented saltire graphs $AS_{n,n}$ and $AS_{n,n+1}$,
and triangular tower graphs $TT_{n,n,n}$;
see \cite{DFv20, FHM19}.
\citet{DFv20} gave an infinite number of families of non-$e$-positive graphs that are not contractible to the claw;
one such family is additionally claw-free, thus establishing that the $e$-positivity is in general not dependent on the existence of an induced claw or of a contraction to a claw.

Note that $e$-positive graphs are Schur positive,
since every element $e_\lambda$ is a linear combination of elements $s_\mu$ over Kostka numbers,
see~\citet[Exercise 2.12]{MR15B}.
\citet{Gas96P} obtained the Schur positivity of the graphs in \cref{conj:ep:cf-inc} by smart bijections, see also~\citet[Theorem 6.3]{SW16}.
\begin{theorem}[\citeauthor{Gas99}]\label{thm:Schur:cf-inc}
Any claw-free incomparability graph is Schur positive.
\end{theorem}

In view of \cref{thm:Schur:cf-inc},
\citet[Conjecture 1.4]{Sta98} proposed the following concise conjecture, and attributed it to \citeauthor{Gas99}.

\begin{conjecture}[\citeauthor{Sta98}]\label{conj:Schur:cf}
Any claw-free graph is Schur positive.
\end{conjecture}

\citet[Propositions 5.3 and 5.4]{Sta95} demonstrated the $e$-positivity of paths and cycles, which can be considered as the most basic graphs in some sense.
\citet{GS01} introduced certain $(e)$-positivity of chromatic symmetric functions of graphs in noncommuting variables, which is stronger than the $e$-positivity in the ordinary sense.
They established the $(e)$-positivity of paths, cycles, and the so-called K-chains.
These results can be used to show the $(e)$-positivity of tadpole graphs, see \citet{LLWY21}.

In the same spirit, we manage to show the $e$-positivity of the line graphs of tadpoles,
which is a bit less basic.
By computing the generating function of the chromatic symmetric functions of these line graphs, one may see that this $e$-positivity is stronger than that of tadpoles, see \cref{coro:e:tadpole}.
On the other hand,
since $e$-positive graphs are all Schur positive, we obtain the Schur positivity of the line graphs of tadpoles.
Recall that \citet{CS05} discovered a characterization of claw-free graphs, in which line graphs is one of six building blocks.
Therefore, the result above confirms a particular case of \cref{conj:Schur:cf}.

This paper is organized as follows. 
In \cref{sec:preliminary}, we give an overview for necessary notion and notation,
and known results that will be of use in the sequel.
\Cref{sec:Line:Tpml} is devoted to the $(e)$-positivity of 
the line graphs of tadpoles.
In \cref{sec:Y:CCm3}, we derive the bivariate generating function of cycle-chord graphs  $\mathrm{CC}_{a,b}$, which is a slight extension of the line graphs.
We also obtain the $(e)$-positivity of~$\mathrm{CC}_{m,3}$.

\section{Preliminaries} \label{sec:preliminary}

This section consists of basic properties of chromatic symmetric functions  given by \citet{Sta95},
and those for the $(e)$-positivity due to \citet{GS01}.
They all will be of use in the sequel.

\begin{proposition}[\citeauthor{Sta95}]\label{prop:csf:disjoint}
For any graphs $G$ and $H$,
\[
X_{G\sqcup H}=X_G X_H,
\]
where $G\sqcup H$ is the disjoint union of $G$ and~$H$. 
\end{proposition}

One way of computing a chromatic symmetric function is by using the power symmetric functions~$p_\lambda$,
 which are defined for integer partitions $\lambda=\lambda_1\lambda_2\dotsm$ by
\[
p_\lambda=p_{\lambda_1}p_{\lambda_2}\dotsm
\quad\text{and}\quad
p_n=\sum_{i\ge 1}x_i^n.
\]

\begin{proposition}[\citeauthor{Sta95}]\label{prop:csf}
For any graph $G$,
\[
X_G
=\sum_{E'\subseteq E}(-1)^{\abs{E'}}p_{\lambda(E')},
\]
where
$\lambda(E')$ 
is the partition consisting of
the component orders of the spanning subgraph~$(V,E')$.
\end{proposition}

The generating function of the power symmetric functions will also be needed (cf.~\cite{Mac95B, MR15B}):
\begin{equation}\label{gf:p}
\sum_{j\ge 0}p_j(-z)^j=\frac{F(z)}{E(z)},
\end{equation}
where
\[
E(z)=\sum_{n\ge0} e_n z^n
\quad\text{and}\quad
F(z)=E(z)-zE'(z).
\]

The generating functions of paths and cycles are known, see \cite[Propositions 5.3 and 5.4]{Sta95}.

\begin{proposition}[\citeauthor{Sta95}]\label{prop:gf:path:cycle}
Denote by $P_n$ the $n$-vertex path and by $C_n$ the $n$-vertex cycle. Then
\[
\sum_{n\geq 0}X_{P_n}z^n
=\frac{E(z)}{F(z)}
\quad\text{and}\quad
\sum_{n\geq 2}X_{C_n}z^n
=\frac{z^2E''(z)}{F(z)}.
\]
As a consequence, paths and cycles are $e$-positive.
\end{proposition}

A beautiful \emph{triple-deletion} property 
can be used to reduce the computation of a chromatic symmetric function,
see \citet[Theorem 3.1, Corollaries 3.2 and 3.3]{OS14}.

\begin{theorem}[\citeauthor{OS14}]\label{thm:rec:3del}
Let $G$ be a graph. Suppose that $G$ contains three verices $u$, $v$, and $w$
such that no two of them are connected by an edge.
Write $e_1=uv$, $e_2=vw$, and $e_3=wu$.
For any set $S\subseteq \{1,2,3\}$, 
denote by $G_S$ the subgraph spanned by the edge set $E(G)\cup\{e_j\colon j\in S\}$.
Then 
\[
X_{G_{12}}=X_{G_1}+X_{G_{23}}-X_{G_3}
\quad\text{and}\quad
X_{G_{123}}=X_{G_{12}}+X_{G_{23}}-X_{G_2}.
\]
\end{theorem}

\citet{GS01} made a systematic study of the algebra of chromatic symmetric functions in noncommuting variables.
Let $G$ be a graph with vertices labeled by $v_{1}, \dots, v_{d}$.
They defined the \emph{chromatic symmetric function in noncommuting variables} of $G$ to be
\[
Y_{G}=\sum_{\kappa} x_{\kappa(v_{1})}\dotsm x_{\kappa(v_{d})},
\]
where the sum runs also over all proper colorings $\kappa$ of $G$.
Note that $Y_{G}$ depends not only on the coloring~$\kappa$, but also on the vertex labeling of $G$.
Let $\Pi_d$ denote the lattice of partitions of the set 
\[
[d]=\{1,\dots, d\}
\]
ordered by refinement. 
Given $\pi \in \Pi_d$, 
the \emph{elementary symmetric function $e_{\pi}$ in noncommuting variables} is defined by
\[
e_{\pi}=\sum_{(i_1, \dots, i_d)} x_{i_1} \dotsm x_{i_d},
\]
where the sum runs over all sequences $(i_1,\dots, i_d)$ of positive integers such that $i_j \neq i_k$ if $j$ and $k$ are in the same block of $\pi$.
The \emph{type} $\lambda(\pi)$ of $\pi$ is 
the integer partition of $d$ whose parts are the block sizes of $\pi$.
It is direct to verify that $e_{\pi}$ becomes $\lambda(\pi)! e_{\lambda(\pi)}$
if we allow the variables $x_i$ to commute, where the symbol $\lambda!$ stands for the product of factorials of all parts of a partition $\lambda$.

Fix an element $i\in[d]$.
\citeauthor{GS01} introduced the following congruence relations:
\begin{itemize}
\item
Two partitions $\pi, \sigma\in \Pi_d$ are
said to be \emph{congruent modulo $i$},
denoted $\pi \equiv_{i} \sigma$, if
\[
\lambda(\pi)=\lambda(\sigma) \quad\text{and}\quad
b_{\pi, i}=b_{\sigma, i},
\]
where $b_{\pi, i}$ is the size of the block of $\pi$ that contains $i$. 
Denote by $(\pi)_i$ the congruence class of $\Pi_d$ modulo $i$ that contains $\pi$.
\item
For any elementary symmetric functions $e_\pi$ and $e_\sigma$,
we write $e_\pi \equiv_i e_\sigma$ if and only if $\pi \equiv_{i} \sigma$.
Denote by $e_{(\pi)_i}$ the congruence class modulo $i$ that contains $e_{\pi}$.
\item
For any graph $G$ on $d$ vertices, we write
\begin{equation}\label{equ:Y_G}
Y_G\equiv_{i} \sum_{(\pi)_i\subseteq \Pi_d} c_{(\pi)_i} e_{(\pi)},
\end{equation}
where
\[
c_{(\pi)_i}=\sum_{\sigma \in(\pi)_i} c_{\sigma}
\quad\text{if}\quad
Y_{G}=\sum_{\sigma \in \Pi_{d}} c_{\sigma} e_{\pi}.
\]
\end{itemize}
They say that $G$ is \emph{$(e)$-positive modulo $i$} if
all coefficients $c_{(\pi)}$ are nonnegative.
When $i$ equals the order $d$ of the graph $G$,
the term ``modulo $i$''
and the letter $i$ in the notation 
\[
b_{\pi, i},\qquad \equiv_i, \qquad\text{and}\qquad (\pi)_i
\]
are all ignored.
For instance, $G$ is said to be \emph{$(e)$-positive} if it is so modulo $d$, and 
\cref{equ:Y_G} reduces to
\begin{equation}\label{eq:equiv:Y}
Y_G\equiv\sum_{(\pi)\subseteq\Pi_d}c_{(\pi)}e_{\pi}.
\end{equation}
Unlike the $e$-coefficients in $X_G$, it is possible that the $e$-coefficients $c_{\pi}$ in~$Y_G$ is not integral.

\cref{prop:(e)toe} builds a bridge between the $(e)$-positivity and the $e$-positivity of the same graph, which is proved and used in~\cite{GS01}.
We state it as a proposition for clarity.
 
\begin{proposition}[\citeauthor{GS01}]\label{prop:(e)toe}
Any $(e)$-positive graph is $e$-positive.
\end{proposition}
\begin{proof}
Suppose that \cref{eq:equiv:Y} holds and $c_{(\pi)}\ge 0$ for all congruence classes $(\pi)$.
Then
\[
X_G=\sum_{(\pi)\subseteq \Pi_d}c_{(\pi)}\lambda(\pi)!\, e_{\lambda(\pi)}
\]
is $e$-positive. 
\end{proof}

\citeauthor{GS01} defined the induction operation $\ind$ on monomials $x_{i_{1}} \dotsm x_{i_{d-1}}$, by
\[
\left(x_{i_{1}} \dotsm x_{i_{d-1}}\right) 
\ind=x_{i_{1}} \dotsm x_{i_{d-2}} x_{i_{d-1}}^{2}
\]
and extended it linearly.
They discovered the following deletion-contraction property to reduce the computation of the chromatic symmetric functions $Y_G$, see \cite[Proposition 3.5]{GS01}.
\begin{proposition}[{\citeauthor{GS01}}]\label{prop:DC}
For the edge $e=v_{d-1}v_d$ in a graph $G$ whose vertices are labeled by $v_1,\dots,v_d$,  
\[
Y_G=Y_{G \backslash e}-Y_{G / e}\ind,
\]
where the contraction of $e$ is labeled by $v_{d-1}$.
\end{proposition}

In practice, we always firstly
select an edge $e$ in \cref{prop:DC} 
according to some reduction strategy,
and then we need to relabel the ends of $e$ so that their labels become the largest and second largest labels.
The effect of such relabelings is demonstrated by
\cref{lem:relabel}, see \cite[Lemma 6.6]{GS01}.

\begin{lemma}[{\citeauthor{GS01}}]\label{lem:relabel}
For any relabeling $\gamma(G)$ of vertices of $G$
that fixes the element $d$, we have
$Y_{\gamma(G)} \equiv Y_G$.
\end{lemma}

For any block $B$ that is disjoint with the set $[d]$,
they use the symbol $\pi/B$ to denote the partition that is formed by adding the block $B$ to $\pi$, 
and the symbol $\pi+(d+1)$ to denote the partition of~$[d+1]$ that is formed by inserting the element $d+1$ into the block of $\pi$ that contains $d$. 
\Cref{prop:ind:e} exhibits the effect of the induction operation acted on $e_\pi$, see \cite[Corollary 6.1]{GS01}.

\begin{proposition}[{\citeauthor{GS01}}]\label{prop:ind:e}
For any partition $\pi\in \Pi_d$, 
\[
e_{\pi}\ind \,\equiv \frac{1}{b_{\pi}}
\brk2{e_{(\pi/d+1)}-e_{(\pi+(d+1))}}.
\]
\end{proposition}
Since the induction respects the congruence relation of partitions, namely,
\[
e_\pi\equiv e_\sigma
\implies
e_\pi\ind \equiv e_\sigma\ind,
\]
it is extendable to congruence classes. Precisely speaking, they defined
\[
e_{(\pi)}\ind\equiv
\sum_{(\sigma)\subseteq \Pi_{d+1}}c_{(\sigma)}e_{(\sigma)}
\quad\text{if}\quad
e_{\pi}\ind=\sum_{\sigma\in \Pi_{d+1}}c_{\sigma}e_{\sigma}.
\]
By using the induction operation, \citet[Propsitions 6.4 and 6.7]{GS01} obtained \cref{prop:path:cycle},
which implies the $(e)$-positivity of paths $P_d$ and cycles $C_d$.

\begin{proposition}[{\citeauthor{GS01}}]\label{prop:path:cycle}
If \cref{eq:equiv:Y} holds for $G=P_d$, 
then 
\begin{align*}
Y_{P_{d+1}}
\equiv
\sum_{(\pi)\subseteq \Pi_d}
\frac{c_{(\pi)}}{b_{\pi}}
\brk2{(b_{\pi}-1)e_{(\pi/d+1)}+
e_{(\pi+(d+1))}}\quad\text{and}\quad
Y_{C_{d+1}}
\equiv \sum_{(\pi)\subseteq \Pi_d}c_{(\pi)}e_{(\pi+(d+1))}.
\end{align*}
\end{proposition}

We also need two more results that allow us to construct new $(e)$-positive graphs from old ones.
\cref{prop:Y:DisjointUnion:G:Km} is a straightforward corollary of 
\cite[Lemma 6.3]{GS01}.

\begin{proposition}\label{prop:Y:DisjointUnion:G:Km}
Let $G$ be a graph on $d$ vertices.
Let $K_m$ be the complete graph of order $m$,
whose vertices are labeled by 
$v_{d+1},\dots, v_{d+m}$. If 
\cref{eq:equiv:Y} holds,
then the disjoint union of $G$ and $K_m$ satisfies
\[
Y_{G\,\uplus K_m}
\equiv
\sum_{(\pi)\subseteq \Pi_d}
c_{(\pi)}e_{(\pi/\{d+1,\,\dots,\,d+m\})}.
\]
\end{proposition}

\cref{thm:(e)-pos:G+K_n} is a restatement of \cite[Theorem 7.6]{GS01}.

\begin{theorem}[{\citeauthor{GS01}}]\label{thm:(e)-pos:G+K_n}
Given $n,d\ge 1$, let $G=(V_1,E_1)$ be a graph with vertex set $V_1=\{v_1, \dots, v_d\}$, and let $G+K_n=(V_2,E_2)$ be the graph with 
\begin{align*}
V_2&=V_1\cup \{v_{d+1},\, \dots,\, v_{d+n-1}\},\\
E_2&=E_1\cup \{v_iv_j\colon i,j\in [d,\, d+n-1]\}.
\end{align*}
If $Y_G$ is $(e)$-positive, 
then so is $Y_{G+K_n}$.
\end{theorem}

\section{The $(e)$-positivity of the line graphs of tadpoles}
\label{sec:Line:Tpml}

The \emph{tadpole} graph $\mathrm{Tp}_{m,l}$
is obtained by identifying a vertex of the cycle $C_m$
with an end of the path $P_l$, see \cref{fig:tadpole}.
Tadpoles are special squid graphs that is investigated by \citet{MMW08}.
We first give the bivariate generating functions for the chromatic symmetric functions of tadpoles and their line graphs. Recall that the functions $E(z)$ and $F(z)$ are defined in \cref{sec:preliminary}.
\begin{figure}[h]
\begin{tikzpicture}[scale=0.3]
\input{tadpole}
\end{tikzpicture}
\caption{The tadpole graph $\mathrm{Tp}_{m,l}$.}\label{fig:tadpole}
\end{figure}

\begin{theorem}\label{thm:gf:tadpole}
The chromatic symmetric functions of tadpole graphs $\mathrm{Tp}_{m,l}$ can be computed by
\begin{equation}\label{rec:TPml}
X_{\mathrm{Tp}_{m,l}}=(m-1)X_{P_{m+l}}-\sum_{i=2}^{m-1}X_{P_{m+l-i}}X_{C_i},
\end{equation}
and their bivariate generating function is
\[
\sum_{m\ge2}\sum_{l\ge0}X_{\mathrm{Tp}_{m,l}}x^my^l
=\frac{x^2}{(x-y)^2}
\brk[s]3{\frac{x(x-y)E''(x)E(y)}{F(x)F(y)}
-y\brk3{\frac{E'(x)}{F(x)}-\frac{E'(y)}{F(y)}}}.
\]
\end{theorem}
\begin{proof}
Let $m\ge 2$ and $l\ge0$. 
By \cref{prop:csf:disjoint,thm:rec:3del}, we obtain
\[
X_{\mathrm{Tp}_{m,l}}
=\begin{cases}
X_{P_{m+l}}+X_{\mathrm{Tp}_{m-1,\,l+1}}-X_{C_{m-1}}X_{P_{l+1}},&\text{if $m\ge 4$},\\
2X_{P_{l+3}}-X_{C_2}X_{P_{l+1}},&\text{if $m=3$},\\
X_{P_{l+2}},&\text{if $m=2$}.
\end{cases}
\]
If $m\ge 4$, then we can use the formula above iteratedly which gives
\[
X_{\mathrm{Tp}_{m,l}}
=(m-3)X_{P_{m+l}}+X_{\mathrm{Tp}_{3,l}}-\sum_{i=3}^{m-1}X_{C_i}X_{P_{m+l-i}},
\]
which simplifies to \cref{rec:TPml}.
It is routine to verify its truth for $m\le 3$.
Using standard techniques of generating functions, 
one may translate \cref{rec:TPml} into the generating function equivalently.
\end{proof}

Let $u$ and $v$ be two adjacent vertices on the cycle $C_m$,
and let $w$ be an end of the path $P_l$.
The line graph $L(\mathrm{Tp}_{m,l})$ of the tadpole $\mathrm{Tp}_{m,l}$ can be obtained
by adding the edges $uw$ and $vw$ to the disjoint union of $C_m$ and $P_l$,
see~\cref{fig:Ltadpole}.
\begin{figure}[htbp]
\begin{tikzpicture}[rotate=90, scale=.3]
\input{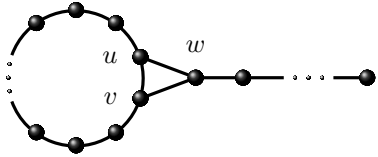}
\end{tikzpicture}
\caption{The line graph of a tadpole.}
\label{fig:Ltadpole}
\end{figure}

\begin{theorem}\label{thm:ep:tadpole:Ltadpole}
The chromatic symmetric functions of the line graphs $L(\mathrm{Tp}_{m,l})$ can be computed by
\begin{equation}\label{fml:csf:Ltadpole}
X_{L(\mathrm{Tp}_{m,l})}
=X_{P_l} X_{C_m}
+2\sum_{k\ge 1} X_{P_{l-k}} X_{C_{m+k}}
-2lX_{P_{m+l}},
\qquad\text{for $m\ge 2$ and $l\ge 0$,}
\end{equation}
or alternatively by 
\begin{equation}\label{rec:Ltadpole:tadpole}
X_{L(\mathrm{Tp}_{m,l})}=2X_{\mathrm{Tp}_{m,l}}-X_{C_m}X_{P_l},\qquad\text{for $m\ge 2$ and $l\ge 0$}.
\end{equation}
Their bivariate generating function is
\[
\sum_{m\ge 2}\sum_{l\ge0}X_{L(\mathrm{Tp}_{m,l})} x^m y^l
=\frac{x^2}{(x-y)^2}
\brk[s]3{\frac{(x^2-y^2)E''(x)E(y)}{F(x)F(y)}
-2y\brk3{\frac{E'(x)}{F(x)}-\frac{E'(y)}{F(y)}}}.
\]
\end{theorem}
\begin{proof}
By \cref{prop:csf:disjoint,thm:rec:3del}, 
we obtain \cref{rec:Ltadpole:tadpole}.
Together with \cref{prop:gf:path:cycle,thm:gf:tadpole}, 
one may deduce \cref{fml:csf:Ltadpole} and the generating function by routine calculation.
\end{proof}

We did not succeed in proving the $e$-positivity of the line graphs $L(\mathrm{Tp}_{m,l})$ by analyzing their generating function.
Yet powerful enough is \cref{prop:(e)toe} to complete this job.

\begin{theorem}
Let $m\ge 3$ and $l\ge 1$.
Let $C_m$ be the cycle $v_1\dotsm v_mv_1$,
and $P_l$ the path $v_{m+1}\dotsm v_{m+l}$.
Then the line graph $L(\mathrm{Tp}_{m,l})$
that is labeled by adding the edges $v_1v_{m+1}$ and $v_mv_{m+1}$ in the disjoint union $C_m\uplus P_l$ is $(e)$-positive.
\end{theorem}

\begin{proof}
We first show it for $l=1$. Let $G=L(\mathrm{Tp}_{m,1})$.
By \cref{prop:DC},  
\begin{equation}\label{pf:Y:L:Tp:m:1}
Y_G
= Y_{\mathrm{Tp}_{m,1}}-Y_{C_m}\ind,
\end{equation}
see \cref{fig:tadpole} with $l=1$ for the labeling of the graph $\mathrm{Tp}_{m,1}$.
By \cref{lem:relabel},
we can relabel the vertices of $\mathrm{Tp}_{m,1}$
so that the resulting graph $\mathrm{Tp}_{m,1}'$ is obtained by adding the edge $v_1v_m$ onto the path $v_1\dotsm v_{m+1}$.
Applying \cref{prop:DC} to $\mathrm{Tp}_{m,1}'$ and by \cref{pf:Y:L:Tp:m:1}, we can infer that
\begin{equation}\label{pf:Y:LTp:m1}
Y_G
\equiv Y_{\mathrm{Tp}_{m,1}'}-Y_{C_m}\ind
\equiv Y_{C_m\,\uplus K_1}-2Y_{C_m}\ind.
\end{equation}
Suppose that
\begin{equation}\label{eq:Y:P:m-1}
Y_{P_{m-1}}
\equiv
\sum_{(\rho)\subseteq \Pi_{m-1}}c_{(\rho)}e_{(\rho)}.
\end{equation}
By \cref{prop:path:cycle}, 
\[
c_{(\rho)}\ge 0
\quad\text{and}\quad
Y_{C_m}\equiv \sum_{(\rho)\subseteq \Pi_{m-1}}c_{(\rho)}e_{(\rho+m)}.
\]
By \cref{prop:Y:DisjointUnion:G:Km}, 
\begin{equation}\label{pf:Cm+K1}
Y_{C_m\,\uplus K_1}
\equiv \sum_{(\rho)\subseteq \Pi_{m-1}}c_{(\rho)}e_{(\rho+m/m+1)}.
\end{equation}
By \cref{prop:ind:e},
\begin{align}
Y_{C_m}\ind
\equiv \sum_{(\rho)\subseteq \Pi_{m-1}}c_{(\rho)}e_{(\rho+m)}\ind
\equiv \sum_{(\rho)\subseteq \Pi_{m-1}} 
\frac{c_{(\rho)}}{b_{\rho}+1}
\brk2{e_{(\rho+m/ m+1)}-e_{(\rho+m+(m+1))}}.\label{pf:Cm:ind}
\end{align}
Therefore, substituting \cref{pf:Cm+K1,pf:Cm:ind} into \cref{pf:Y:LTp:m1}, we obtain
\[
Y_G
\equiv
\sum_{(\rho)\subseteq \Pi_{m-1}}
\frac{c_{(\rho)}}{b_{\rho}+1}
\brk2{(b_{\rho}-1)e_{(\rho +m/ m+1)}+2e_{(\rho +m+(m+1))}}.
\]
Since $b_{\rho}\ge 1$ for all $\rho\in \Pi_{m-1}$, we obtain that $Y_G$ is $(e)$-positive.

Applying \cref{thm:(e)-pos:G+K_n} recursively yields that
$Y_{L(\mathrm{Tp}_{m,l})}$ is $(e)$-positive.
\end{proof}

\begin{corollary}\label{coro:LTp}
The line graphs $L(\mathrm{Tp}_{m,l})$ are $e$-positive for any integers $m\ge 3$ and $l\ge 1$.
\end{corollary}
\begin{proof}
Immediate by \cref{prop:(e)toe}.
\end{proof}

\begin{corollary}\label{coro:e:tadpole}
The tadpoles $\mathrm{Tp}_{m,l}$ are $e$-positive for any integers $m\ge 3$ and $l\ge 1$.
\end{corollary}
\begin{proof}
Immediate by \cref{rec:Ltadpole:tadpole,coro:LTp}.
\end{proof}

In fact, tadpoles labeled as in \cref{fig:tadpole} are $(e)$-positive, which is a
corollary of \cref{prop:path:cycle,thm:(e)-pos:G+K_n}, see also~\cite{LLWY21}.

\section{The $(e)$-positivity of the cycle-chord graphs $\mathrm{CC}_{m,3}$}
\label{sec:Y:CCm3}

Regarding the line graphs $L(\mathrm{Tp}_{m,1})$ as obtained by identifying an edge of the cycles $C_m$ with an edge of the triangle $C_3$,
we extend this $(e)$-positive graph class a bit by considering the graphs that are obtained by identifying an edge in $C_m$ with an edge of the rectangle $C_4$.
In more general, one may consider the \emph{cycle-chord graph} $\mathrm{CC}_{a,b}$ that is obtained by 
identifying an edge of $C_{a+1}$ with an edge of~$C_{b+1}$, where $a,b\ge1$.
In particular, the graph
$\mathrm{CC}_{m,2}$ is $L(\mathrm{Tp}_{m,1})$.
\begin{figure}[htbp]
\begin{tikzpicture}[rotate=90, scale=.4]
\input{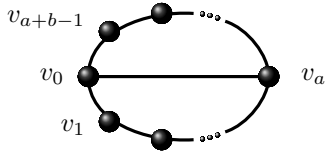}
\end{tikzpicture}
\caption{The cycle-chord graph $\mathrm{CC}_{a,b}$.}
\label{fig:CCml}
\end{figure}

In this section, we first present the bivariate generating function for the chromatic symmetric functions of the cycle-chord graphs $\mathrm{CC}_{a,b}$, and then
show the $(e)$-positivity of $\mathrm{CC}_{m,3}$.

\begin{theorem}\label{thm:gf:cycle-chord}
The generating function of cycle-chord graphs $\mathrm{CC}_{a,b}$ is
\[
\sum_{a,b\ge1}X_{\mathrm{CC}_{a,b}}x^ay^b
=\frac{xy}{(x-y)^2}\brk[s]3{
\frac{x^2E''(x)E(y)+y^2E''(y)E(x)}{F(x)F(y)}
-\frac{2xy}{x-y}\brk3{
\frac{E'(x)}{F(x)}-\frac{E'(y)}{F(y)}}}.
\]
\end{theorem}
\begin{proof}
Let $\mathrm{CC}_{a,b}=(V,E)$, where $V=\{v_0,v_1, \dotsc, v_{a+b-1}\}$
and $E=E_a\cup E_b\cup \{v_0v_a\}$ with
\[
E_a=\{v_0v_1,\,v_1v_2,\,\dotsc, v_{a-1}v_a\}
\quad\text{and}\quad
E_b=\{v_av_{a+1},\,v_{a+1}v_{a+2},\,\dotsc,\,v_{a+b-1}v_0\}.
\]
Since removing the edge $v_0v_a$ from $\mathrm{CC}_{a,b}$ results in 
the cycle $C_{a+b}$, by \cref{prop:csf},
\[
\sum_{v_0v_a\not\in S\subseteq E}(-1)^{\abs{S}}p_{\lambda(S)}
=X_{C_{a+b}}.
\]
When $v_0v_a\in S$, 
let $M$ be the component of the graph $(V,S)$
that contains the edge $v_0v_a$.
Let $a'$ be the number of edges in $M\cap E_a$
and $b'$ the number of edges in $M\cap E_b$.
Then $0\le a'\le a$ and $0\le b'\le b$.
According to whether $a'=a$ and $b'=b$, we obtain 
\begin{align}
X_{\mathrm{CC}_{a,b}}
&=X_{C_{a+b}}
+\sum_{a'=0}^{a-1}\sum_{b'=0}^{b-1}
(-1)^{a'+b'+1}(a'+1)(b'+1)p_{a'+b'+2}
X_{P_{a-a'-1}}X_{P_{b-b'-1}}
+(-1)^{a+b+1}p_{a+b}\notag\\
&\quad
+\sum_{a'=0}^{a-1}(-1)^{a'+b+1}(a'+1)p_{a'+b+1}X_{P_{a-a'-1}}
+\sum_{b'=0}^{b-1}(-1)^{a+b'+1}(b'+1)p_{a+b'+1}X_{P_{b-b'-1}}
\notag\\
&=X_{C_{a+b}}
+\sum_{i=1}^{a}\sum_{j=1}^{b}
(-1)^{i+j+1}\cdotp i\cdotp j\cdotp p_{i+j}X_{P_{a-i}}X_{P_{b-j}}
+(-1)^{a+b+1}p_{a+b}\notag\\
&\quad
+\sum_{i=1}^{a}(-1)^{b+i}\cdotp i\cdotp p_{b+i}X_{P_{a-i}}
+\sum_{j=1}^{b}(-1)^{a+j}\cdotp j\cdotp p_{a+j}X_{P_{b-j}}.
\label{rec:CC:a:b}
\end{align}
Considering the generating function
\[
H(x,y)=\sum_{a,b\ge1}X_{\mathrm{CC}_{a,b}}x^ay^b,
\]
we compute each term 
on the right side of \cref{rec:CC:a:b}
by multiplying $x^ay^b$ and summing over $a,b\ge1$.
For convenience, we will use the substitutions 
\[
s=-x,\quad
t=-s,
\quad\text{and}\quad
Q_z=\frac{F(z)}{E(z)}.
\]
For the first and third term in \cref{rec:CC:a:b}, 
by \cref{prop:gf:path:cycle,gf:p},
\begin{align}
\sum_{a,b\ge1}X_{C_{a+b}}x^a y^b
&=\frac{xy}{x-y}\sum_{h\ge2}X_{C_h}(x^{h-1}-y^{h-1})
=\frac{xy}{x-y}\brk3{\frac{xE''(x)}{F(x)}-\frac{yE''(y)}{F(y)}},
\label{gf:C:a+b}\\
\sum_{a,b\ge1}(-1)^{a+b+1}p_{a+b}x^ay^b
&=-\sum_{a,b\ge1}p_{a+b}s^a t^b
=\frac{xy}{x-y}\sum_{h\ge2}p_h(s^{h-1}-t^{h-1})\notag\\
&=\frac{xy}{x-y}
\brk3{\frac{Q_x-1-e_1 s}{s}-\frac{Q_y-1-e_1 t}{t}}
=\frac{xQ_y-yQ_x}{x-y}-1.\notag
\end{align}
For the fourth term in \cref{rec:CC:a:b},
\begin{align*}
&\sum_{a,b\ge1}\sum_{i=1}^{a}
(-1)^{b+i}\cdotp i\cdotp p_{b+i}X_{P_{a-i}}x^a y^b
=\sum_{b,i\ge1}(-1)^{i}\cdotp i\cdotp p_{b+i}\cdotp t^b
\sum_{a\ge i}X_{P_{a-i}}x^a\\
=&\frac{1}{Q_x}\sum_{b,i\ge1} i\cdotp p_{b+i}\cdotp s^i t^b
=\frac{1}{Q_x}\sum_{h\ge2} 
p_h\cdotp\brk2{st^{h-1}+2s^2t^{h-2}+\dotsb+(h-1)s^{h-1}t}\\
=&\frac{xy}{(x-y)^2Q_x}
\sum_{h\ge2} p_h\cdotp\brk2{hs^h-hs^{h-1}t-s^h+t^h}
=\frac{xy\brk[s]1{(x-y)Q_x'-(Q_x-Q_y)}}{(x-y)^2Q_x}.
\end{align*}
Exchanging $x$ and $y$ and exchanging $a$ and $b$ in the formula above, we obtain 
\[
\sum_{a,b\ge1}\sum_{j=1}^{b}
(-1)^{a+j}\cdotp j\cdotp p_{a+j}X_{P_{b-j}}x^a y^b
=\frac{xy\brk[s]1{(y-x)Q_y'+(Q_x-Q_y)}}{(x-y)^2Q_y}.
\]
For the second term in \cref{rec:CC:a:b}, we compute
\begin{align*}
&\sum_{a,b\ge1}
\sum_{i=1}^{a}\sum_{j=1}^{b}
(-1)^{i+j+1}\cdotp i\cdotp j\cdotp p_{i+j}X_{P_{a-i}}X_{P_{b-j}}
x^a y^b\\
=&-\sum_{i,j\ge1}i\cdotp j\cdotp p_{i+j}
\sum_{a\ge i}X_{P_{a-i}}s^a 
\sum_{b\ge j}X_{P_{b-j}}t^b\\
=&-\frac{1}{Q_x Q_y}
\sum_{i,j\ge1}i\cdotp j\cdotp p_{i+j}\cdotp s^i\cdotp t^j
=-\frac{1}{Q_x Q_y}\sum_{h\ge2}p_h 
\sum_{i=1}^{h-1} i\cdotp (h-i)\cdotp s^i t^{h-i}\\
=&-\frac{st}{(s-t)^3Q_x Q_y}\sum_{h\ge2}p_h 
\brk[s]2{(h-1)\brk1{s^{h+1}-t^{h+1}}-(h+1)st\brk1{s^{h-1}-t^{h-1}}}\\
=&\frac{xy}{(x-y)^3Q_x Q_y}
\brk[s]2{(x+y)(Q_x-Q_y)-(x-y)(xQ_x'+yQ_y')}.
\end{align*}
Adding the five results above up, we obtain
\begin{align*}
H(x,y)
&=\frac{xy}{x-y}\brk3{\frac{xE''(x)}{F(x)}-\frac{yE''(y)}{F(y)}}
+\frac{xQ_y-yQ_x}{x-y}-1\\
&\qquad+\frac{xy\brk[s]1{(x-y)Q_x'-(Q_x-Q_y)}}{(x-y)^2Q_x}
+\frac{xy\brk[s]1{(y-x)Q_y'+(Q_x-Q_y)}}{(x-y)^2Q_y}\\
&\qquad+\frac{xy}{(x-y)^3Q_x Q_y}
\brk[s]2{(x+y)(Q_x-Q_y)-(x-y)(xQ_x'+yQ_y')},
\end{align*}
which simplifies to the desired formula.
\end{proof}

In the procedure of showing the $(e)$-positivity of $Y_{\mathrm{CC}_{m,3}}$,
we will use \cref{prop:DC,lem:relabel} frequently in which we need to relabel the largest vertex. We write
$(d, i)$ to denote the transposition of the elements $d$ and $i$.
For any partition $\pi\in\Pi_d$ and any graph $G$ on $d$ vertices, we write
\begin{itemize}
\item
$\pi\circ (d, i)$ to denote the partition of $[d]$ that interchanges the elements $d$ and $i$ in $\pi$,
\item
$G\circ (d, i)$ to denote the relabeling of $G$ that interchanges the vertex labels $v_d$ and $v_{i}$, and 
\item
$Y_G\circ (d, i)$ to denote the chromatic symmetric function $Y_{G\circ (d, i)}$.
\end{itemize}
It follows that
\[
Y_G=\sum_{\pi \in \Pi_d} c_{\pi}e_{\pi}\quad
\Longrightarrow\quad
Y_G\circ (d, i)=\sum_{\pi \in \Pi_d} c_{\pi}e_{\pi\circ (d, i)}.
\]
Another basic result of the transposition is \cref{prop:ind:relabel:d}.

\begin{proposition}\label{prop:ind:relabel:d}
For any partition $\pi \in \Pi_d$ and any element $i\in [d]$,
\[
e_{\pi\circ (d, i)}\ind \,
\equiv 
\frac{1}{b_{\pi, i}}\brk2{e_{(\pi/d+1)}-e_{(\pi+_{i}(d+1))}}\circ (d, i).
\]
where
$\pi+_i(d+1)$ is the partition of $[d+1]$ formed by inserting the element $d+1$ into the block of $\pi$ that contains $i$.
\end{proposition}
\begin{proof}
By \cref{prop:ind:e}, we have
\begin{align*}
e_{\pi\circ (d, i)}\ind
&\equiv
\frac{1}{b_{\pi\circ (d, i)}}\brk2{e_{(\pi\circ (d, i)/d+1)}-e_{(\pi\circ (d, i)+(d+1))}},
\end{align*}
which simplifies to the desired congruence relation.
\end{proof}

Note that \cref{prop:ind:relabel:d} reduces to \cref{prop:ind:e} when $i=d$.

Here comes the main result of \cref{sec:Y:CCm3}.
\begin{theorem}\label{thm:Y:CCm3}
For any integer $m\ge 3$, 
the cycle-chord graph $\mathrm{CC}_{m,3}$ that is labeled by adding the edge~$v_2v_{m+1}$ to the cycle $v_1\dotsm v_{m+2}v_1$ is $(e)$-positive.
\end{theorem}
\begin{proof}
By \cref{prop:DC},
\begin{align*}
Y_{\mathrm{CC}_{m,3}}=
Y_{\mathrm{Tp}_{m,2}}-Y_{L(\mathrm{Tp}_{m,1})}\ind,
\end{align*}
see the left part of \cref{fig:CCm3} for the vertex labeling of $\mathrm{CC}_{m,3}$. 
\begin{figure}[htbp]
\begin{tikzpicture}[rotate=90, scale=.4]
\input{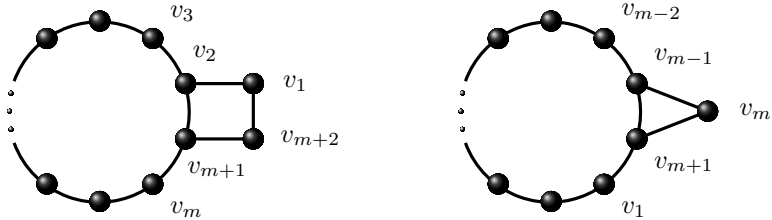}
\end{tikzpicture}
\caption{The cycle-chord graph $\mathrm{CC}_{m,3}$ and the line graph $Y_{L(\mathrm{Tp}_{m,1})'}$.}
\label{fig:CCm3}
\end{figure}

We relabel the vertices of the tadpole $\mathrm{Tp}_{m,2}$
so that it is obtained by adding the edge $v_1v_m$ in the path $v_1\dotsm v_{m+2}$.
Write the resulting graph as $\mathrm{Tp}_{m,2}'$. 
We also relabel the vertices of the line graph $L(\mathrm{Tp}_{m,1})$
so that it is obtained by adding the edge $v_{m-1}v_{m+1}$ in the cycle $v_1\dotsm v_{m+1}v_1$.
Write the resulting graph as~$L(\mathrm{Tp}_{m,1})'$, see the right part of \cref{fig:CCm3}.
By \cref{lem:relabel} and by applying \cref{prop:DC} another twice, we can deduce that
\begin{align}
Y_{\mathrm{CC}_{m,3}}
&\equiv_{m+2} 
Y_{\mathrm{Tp}_{m,2}'}-Y_{L(\mathrm{Tp}_{m,1})'}\ind\notag\\
&\equiv_{m+2} 
\brk1{Y_{\mathrm{Tp}_{m,1}\,\uplus K_1}-Y_{\mathrm{Tp}_{m,1}}\ind}-Y_{L(\mathrm{Tp}_{m,1})'}\ind\notag\\
&\equiv_{m+2} 
Y_{\mathrm{Tp}_{m,1}\,\uplus K_1}-(Y_{C_m\,\uplus K_1}-Y_{C_m}\ind)\ind-Y_{L(\mathrm{Tp}_{m,1})'}\ind\notag\\
&\equiv_{m+2} 
Y_{\mathrm{Tp}_{m,1}\,\uplus K_1}+Y_{C_m}\ind \ind-Y_{C_m\,\uplus K_1}\ind-Y_{L(\mathrm{Tp}_{m,1})'}\ind.\label{pf:Y:CCm2}
\end{align}
We proceed by computing the four terms in \cref{pf:Y:CCm2} separately.
In order to compute the last term,
we pause here and give a lemma for computing 
$Y_{L(\mathrm{Tp}_{m,1})'}$.

\begin{lemma}\label{lem:lineTp}
If \cref{eq:Y:P:m-1} holds,
then
\begin{equation}\label{eq:Y:LTpm1}
Y_{L(\mathrm{Tp}_{m,1})'}
\equiv
\sum_{(\rho)\subseteq \Pi_{m-1}}c_{(\rho)}
\brk3{
\frac{2}{b_{\rho}+1}e_{(\rho+m+(m+1))}-
\frac{b_{\rho}-1}{b_{\rho}(b_{\rho}+1)}e_{(\rho+m/m+1)}
+\frac{b_{\rho}-1}{b_{\rho}}
e_{(\rho')}},
\end{equation}
where $\rho'\in\Pi_{m+1}$ is the partition
$\rho\backslash \{m-1\}+m+(m+1)/m-1$.
\end{lemma}

\begin{proof}[Proof of \cref{lem:lineTp}]
By using \cref{prop:DC}, we obtain
\[
Y_{C_m}=Y_{P_m}-Y_{C_{m-1}}\ind.
\]
It follows that
\begin{align}\label{pf:YCm-1}
Y_{C_{m-1}}\ind\ind=Y_{P_m}\ind-Y_{C_m}\ind.
\end{align}
Let $i\in [m-1]$. It will be carefully set to facilliate the computation in the sequel. By using \cref{prop:DC} thrice and using \cref{pf:YCm-1}, we can deduce
\begin{align}
Y_{L(\mathrm{Tp}_{m,1})'}
&\equiv
Y_{\mathrm{Tp}_{m,1}\circ (m, m-1)}-Y_{C_m}\ind\notag\\
&\equiv
Y_{P_{m+1}}-Y_{\mathrm{Tp}_{m-1,1}}\ind-Y_{C_m}\ind\notag\\
&\equiv
Y_{P_{m+1}}-\brk1{Y_{(C_{m-1}\,\uplus K_1)\circ (m, i)}-Y_{C_{m-1}}\ind} \ind-Y_{C_m}\ind\notag\\
&\equiv
Y_{P_{m+1}}-Y_{C_{m-1}\,\uplus K_1}\circ (m, i)\ind+Y_{P_m}\ind-2Y_{C_m}\ind, \label{pf:Y:LTpm1}
\end{align}
where the graph $C_{m-1}\,\uplus K_1$ is labeled by $C_{m-1}=v_1\dotsm v_{m-1}v_1$ and $V(K_1)=\{v_m\}$.
We will calculate the four terms in \cref{pf:Y:LTpm1} independently.
Suppose that
\begin{align*}
Y_{P_{m-2}}
&\equiv
 \sum_{(\tau)\subseteq \Pi_{m-2}}h_{(\tau)}e_{(\tau)},\\
Y_{P_{m-1}}
&\equiv
\sum_{(\rho)\subseteq \Pi_{m-1}}c_{(\rho)}e_{(\rho)}\quad \text{and}\\
Y_{C_{m-1}}
&\equiv
\sum_{(\rho)\subseteq \Pi_{m-1}}c'_{(\rho)}e_{(\rho)}.
\end{align*}
Then by \cref{prop:path:cycle}, 
\begin{align*}
c_{(\rho)}&=
\begin{cases}
\displaystyle\frac{h_{(\tau)}}{b_{\tau}}(b_{\tau}-1),&\text{if $\rho=\tau/(m-1)$},\\[10pt]
\displaystyle\frac{h_{(\tau)}}{b_{\rho}-1},&\mbox{if $\rho=\tau+(m-1)$},\\[9pt]
0,&\mbox{otherwise},
\end{cases}\\
c'_{(\rho)}&=
\begin{cases}
h_{(\tau)},&\text{if $\rho=\tau+(m-1)$},\\[9pt]
0,&\text{otherwise}.
\end{cases}
\end{align*}
It follows that 
$c'_{(\rho)}=h_{(\tau)}=(b_{\rho}-1)c_{(\rho)}$,
if $\rho=\tau+(m-1)$.
By using \cref{prop:Y:DisjointUnion:G:Km,prop:ind:relabel:d}, we obtain
\begin{align*}
Y_{C_{m-1}\,\uplus K_1}\circ (m, i)\ind
\equiv
&\sum_{(\rho)\subseteq \Pi_{m-1}}c'_{(\rho)}e_{(\rho/ m)}\circ (m, i)\ind\\
\equiv
&\sum_{(\rho)\subseteq \Pi_{m-1}}\frac{c'_{(\rho)}}{b_{\rho/m,\,i}}\brk2{e_{(\rho/ m/m+1))}-e_{\rho/m+_{i}(m+1)}} \circ (m, i)\\
\equiv
&\sum_{(\rho)\subseteq \Pi_{m-1}}
\frac{(b_{\rho}-1)c_{(\rho)}}{b_{\rho,i}}\brk2{e_{(\rho /m/m+1)}-e_{\rho+_{i}(m+1)/m}}\circ (m, i).
\end{align*}
Taking $i=m-1$, we obtain
\begin{align}
Y_{C_{m-1}\,\uplus K_1}\circ (m, i)\ind
\equiv
&\sum_{(\rho)\subseteq \Pi_{m-1}}
\frac{(b_{\rho}-1)c_{(\rho)}}{b_{\rho}}
\brk2{e_{(\rho /m/m+1)}-e_{(\rho')}}.\label{pf:Y:Cm-1K1:ind}
\end{align}
On the other hand, by \cref{prop:path:cycle}, we obtain 
\begin{align}\label{pf2:Pm}
Y_{P_m}
&\equiv
\sum_{(\rho)\subseteq \Pi_{m-1}}
\frac{c_{(\rho)}}{b_{\rho}}
\brk2{(b_{\rho}-1)e_{(\rho/m)}+e_{(\rho+m)}}\quad\text{and}\\
Y_{C_m}\label{pf2:Cm}
&\equiv
\sum_{(\rho)\subseteq \Pi_{m-1}}
c_{(\rho)}e_{(\rho+m)}.
\end{align}
By \cref{prop:path:cycle}, we obtain
\begin{align}
Y_{P_{m+1}}
\equiv
&\sum_{(\rho)\subseteq \Pi_{m-1}}
\frac{(b_{\rho}-1)c_{(\rho)}}{b_{\rho}b_{\rho/m}}
\brk3{(b_{\rho/m}-1)e_{(\rho/m/m+1)}+e_{(\rho/m+(m+1))}}\notag\\
&\qquad+\sum_{(\rho)\subseteq \Pi_{m-1}}
\frac{c_{(\rho)}}{b_{\rho}b_{\rho+m}}
\brk3{(b_{\rho+m}-1)e_{(\rho+m/m+1)}+e_{(\rho+m+(m+1))}}\notag\\
\equiv
&\sum_{(\rho)\subseteq \Pi_{m-1}}c_{(\rho)}
\brk3{\frac{b_{\rho}-1}{b_{\rho}}e_{(\rho/\{m,m+1\})}+
\frac{1}{b_{\rho}+1}e_{(\rho+m/m+1)}+
\frac{1}{b_{\rho}(b_{\rho}+1)}e_{(\rho+m+(m+1))}}.\label{pf:Y:Pm+1}
\end{align}
Applying the induction to \cref{pf2:Pm,pf2:Cm} respectively, and by \cref{prop:ind:e}, we obtain
\begin{align}
Y_{P_m}\ind
\equiv
&\sum_{(\rho)\subseteq \Pi_{m-1}}
\frac{c_{(\rho)}}{b_{\rho}}
\brk2{(b_{\rho}-1)e_{(\rho/m)}\ind +
e_{(\rho+m)}\ind}\notag\\
\equiv
&\sum_{(\rho)\subseteq \Pi_{m-1}}
\frac{c_{(\rho)}}{b_{\rho}}
(b_{\rho}-1)
\brk2{e_{(\rho/m/m+1)}-e_{(\rho/\{m,m+1\})}}\notag\\
&\qquad+\sum_{(\rho)\subseteq \Pi_{m-1}}
\frac{c_{(\rho)}}{b_{\rho}(b_{\rho}+1)}
\brk2{e_{(\rho+m/m+1)}-e_{(\rho+m+(m+1))}},
\quad \text{and}\label{pf:Y:Pm:ind}\\
Y_{C_m}\ind
\equiv 
&\sum_{(\rho)\subseteq \Pi_{m-1}}
c_{(\rho)}e_{(\rho+m)}\ind
\equiv
\sum_{(\rho)\subseteq \Pi_{m-1}}\frac{c_{(\rho)}}{b_{\rho}+1}
\brk2{e_{(\rho+m/m+1)}-e_{(\rho+m+(m+1))}}.\label{pf:Y:Cm:ind}
\end{align}
Therefore, substituting \cref{pf:Y:Cm-1K1:ind,pf:Y:Pm+1,pf:Y:Pm:ind,pf:Y:Cm:ind} into \cref{pf:Y:LTpm1} with $i=m-1$, we obtain
\begin{align*}
Y_{L(\mathrm{Tp}_{m,1})'}
\equiv
&Y_{P_{m+1}}-Y_{C_{m-1}\,\uplus K_1}\circ (m, m-1)\ind+Y_{P_m}\ind-2Y_{C_m}\ind\\
\equiv
&\sum_{(\rho)\subseteq \Pi_{m-1}}c_{(\rho)}
\brk2{
\frac{b_{\rho}-1}{b_{\rho}}e_{(\rho/m/m+1)}+
\frac{2}{b_{\rho}+1}e_{(\rho+m+(m+1))}-
\frac{b_{\rho}-1}{b_{\rho}(b_{\rho}+1)}e_{(\rho+m/m+1)}
}\\
&-\sum_{(\rho)\subseteq \Pi_{m-1}}
\frac{(b_{\rho}-1)c_{(\rho)}}{b_{\rho}}
\brk2{e_{(\rho /m/m+1)}-e_{(\rho')}},
\end{align*}
which simplifies to \cref{eq:Y:LTpm1}.
\end{proof}

Now we are in a position to complete the proof of 
\cref{thm:Y:CCm3}.
Applying the induction to \cref{eq:Y:LTpm1}, and by  \cref{prop:ind:e}, we can compute
\begin{align*}
Y_{L(\mathrm{Tp}_{m,1})'}\ind
\equiv
&\sum_{(\rho)\subseteq \Pi_{m-1}}c_{(\rho)}
\brk3{
\frac{2}{b_{\rho}+1}e_{(\rho+m+(m+1))}\ind-
\frac{b_{\rho}-1}{b_{\rho}(b_{\rho}+1)}e_{(\rho+m/m+1)}\ind
+\frac{b_{\rho}-1}{b_{\rho}}e_{(\rho')}\ind}
\\
\equiv
&\sum_{(\rho)\subseteq \Pi_{m-1}}c_{(\rho)}
\biggl(
\frac{2}{(b_{\rho}+1)(b_{\rho}+2)}
\brk2{e_{(\rho+m+(m+1)/m+2)}
-e_{(\rho+m+(m+1)+(m+2))}}\\
&\qquad\qquad-\frac{b_{\rho}-1}{b_{\rho}(b_{\rho}+1)}
\brk2{e_{(\rho+m/m+1/m+2)}
-e_{(\rho+m/(m+1)+(m+2))}}\\
&\qquad\qquad+\frac{b_{\rho}-1}{b_{\rho}(b_{\rho}+1)}\brk2{e_{(\rho'/m+2)}
-e_{(\rho'+(m+2))}}\biggr)\\
\equiv
&\sum_{(\rho)\subseteq \Pi_{m-1}}c_{(\rho)}
\biggl(
\frac{2}{(b_{\rho}+1)(b_{\rho}+2)}\brk2{e_{(\rho+m+(m+1)/m+2)}-e_{(\rho+m+(m+1)+(m+2))}}\\
&\qquad\qquad+\frac{b_{\rho}-1}{b_{\rho}(b_{\rho}+1)}\brk2{e_{(\rho+m/(m+1)+(m+2))}-e_{(\rho'+(m+2))}}\biggr).
\end{align*}

Suppose that
\cref{eq:Y:P:m-1} holds.
By \cref{prop:Y:DisjointUnion:G:Km,pf2:Cm}, we obtain
\begin{equation}\label{Y:CmK1}
Y_{C_m\,\uplus K_1}
\equiv
\sum_{(\rho)\subseteq \Pi_{m-1}}
c_{(\rho)}e_{(\rho+m/m+1)}.
\end{equation}
Applying \cref{prop:DC} to the graph $\mathrm{Tp}_{m,1}$, 
and by using \cref{Y:CmK1,pf:Y:Cm:ind}, we can deduce that
\begin{align*}
Y_{\mathrm{Tp}_{m,1}}
&=
Y_{C_m\,\uplus K_1}-Y_{C_m}\ind\\
&\equiv
\sum_{(\rho)\subseteq \Pi_{m-1}}
c_{(\rho)}e_{(\rho+m/m+1)}
-\sum_{(\rho)\subseteq \Pi_{m-1}}\frac{c_{(\rho)}}{b_{\rho}+1}
\brk2{e_{(\rho+m/m+1)}-e_{(\rho+m+(m+1))}}\\
&\equiv
\sum_{(\rho)\subseteq \Pi_{m-1}}
\frac{c_{(\rho)}}{b_{\rho}+1}
\brk2{b_{\rho}e_{(\rho+m/m+1)}+
e_{(\rho+m+(m+1))}}.
\end{align*}
It follows by \cref{prop:Y:DisjointUnion:G:Km} that
\begin{align*}
Y_{\mathrm{Tp}_{m,1}\,\uplus K_1}
\equiv
\sum_{(\rho)\subseteq \Pi_{m-1}}
\frac{c_{(\rho)}}{b_{\rho}+1}
\brk2{b_{\rho}e_{(\rho+m/m+1/m+2)}+
e_{(\rho+m+(m+1)/m+2)}}.
\end{align*}
Applying the induction to \cref{Y:CmK1,pf:Y:Cm:ind} respectively, and by \cref{prop:ind:e}, we c an infer that
\begin{align*}
Y_{C_m}\ind\ind
\equiv
&\sum_{(\rho)\subseteq \Pi_{m-1}}
\frac{c_{(\rho)}}{b_{\rho}+1}
\brk2{e_{(\rho+m/m+1)}\ind-e_{(\rho+m+(m+1))}\ind}\\
\equiv
&\sum_{(\rho)\subseteq \Pi_{m-1}}
\brk2{e_{(\rho+m/m+1/m+2)}-e_{(\rho+m/(m+1)+(m+2))}}\\
&-\sum_{(\rho)\subseteq \Pi_{m-1}}
\frac{c_{(\rho)}}{(b_{\rho}+1)(b_{\rho}+2)}
\brk2{e_{(\rho+m+(m+1)/m+2)}-e_{(\rho+m+(m+1)+(m+2))}}\quad\text{and}
\\
Y_{C_m\,\uplus K_1}\ind
\equiv
&\sum_{(\rho)\subseteq \Pi_{m-1}}
c_{(\rho)}e_{(\rho+m/m+1)}\ind\\
\equiv
&\sum_{(\rho)\subseteq \Pi_{m-1}}
c_{(\rho)}
\brk2{e_{(\rho+m/m+1/m+2)}-e_{(\rho+m/(m+1)+(m+2))}}.
\end{align*}
Now we can continue calculating $Y_{\mathrm{CC}_{m,3}}$ by \cref{pf:Y:CCm2} as follows.
\begin{align*}
Y_{\mathrm{CC}_{m,3}}
&\equiv
Y_{\mathrm{Tp}_{m,1}\,\uplus K_1}+Y_{C_m}\ind \ind-Y_{C_m\,\uplus K_1}\ind-Y_{L(\mathrm{Tp}_{m,1})'}\ind\\
&\equiv\sum_{(\rho)\subseteq \Pi_{m-1}}
\frac{c_{(\rho)}}{b_{\rho}+1}
\brk2{b_{\rho}e_{(\rho+m/m+1/m+2)}+
e_{(\rho+m+(m+1)/m+2)}}\\
&\qquad+\sum_{(\rho)\subseteq \Pi_{m-1}}
\frac{c_{(\rho)}}{b_{\rho}+1}
\brk2{e_{(\rho+m/m+1/m+2)}-e_{(\rho+m/(m+1)+(m+2))}}\\
&\qquad-\sum_{(\rho)\subseteq \Pi_{m-1}}
\frac{c_{(\rho)}}{(b_{\rho}+1)(b_{\rho}+2)}
\brk2{e_{(\rho+m+(m+1)/m+2)}-e_{(\rho+m+(m+1)+(m+2))}}\\
&\qquad-\sum_{(\rho)\subseteq \Pi_{m-1}}
c_{(\rho)}
\brk2{e_{(\rho+m/m+1/m+2)}-e_{(\rho+m/(m+1)+(m+2))}}\\
&\qquad-\sum_{(\rho)\subseteq \Pi_{m-1}}c_{(\rho)}
\biggl(
\frac{2}{(b_{\rho}+1)(b_{\rho}+2)}\brk2{e_{(\rho+m+(m+1)/m+2)}-e_{(\rho+m+(m+1)+(m+2))}}\\
&\qquad\qquad\qquad+\frac{b_{\rho}-1}{b_{\rho}(b_{\rho}+1)}\brk2{e_{(\rho+m/(m+1)+(m+2))}-e_{(\rho'+(m+2))}}\biggr)\\
&\equiv
\sum_{(\rho)\subseteq \Pi_{m-1}}
\frac{c_{(\rho)}}{b_{\rho}+1}
\biggl(
\frac{b_{\rho}-1}{b_{\rho}+2}
e_{(\rho+m+(m+1)/m+2)}+
\frac{b_{\rho}^2-b_{\rho}+1}{b_{\rho}}
e_{(\rho+m/(m+1)+(m+2))}\\
&\qquad\qquad\qquad+
\frac{3}{b_{\rho}+2}
e_{(\rho+m+(m+1)+(m+2))}
+\frac{b_{\rho}-1}{b_{\rho}}
e_{(\rho'+(m+2))}\biggr).
\end{align*}
Since $b_{\rho}\ge 1$ for all $\rho\in \Pi_{m-1}$, we obtain that $Y_{\mathrm{CC}_{m,3}}$ is $(e)$-positive.
\end{proof}

\begin{corollary}
The cycle-chord graphs $\mathrm{CC}_{m,3}$ are $e$-positive for any integer $m\ge 3$.
\end{corollary}
\begin{proof}
Immediate by \cref{prop:(e)toe,thm:Y:CCm3}.
\end{proof}

\bibliography{../csf}

\end{document}